\documentclass[a4paper,reqno,12pt]{amsart}

\usepackage[utf8]{inputenc}
\usepackage{amsmath, amssymb, amsthm, epsfig}
\usepackage{hyperref, latexsym}
\usepackage{url}
\usepackage[mathscr]{euscript}
\usepackage{color}
\usepackage{harpoon}
\usepackage{mathabx}
\usepackage{tikz}
\usepackage{mathtools,fbox}
\usepackage{refcheck}
\DeclareMathAlphabet{\pazocal}{OMS}{zplm}{m}{n}

\usepackage{pgfplotstable}
\usepackage{pgfplots}

\usepackage{fullpage} 
\usepackage[nodisplayskipstretch]{setspace}
\mathtoolsset{showonlyrefs}
\norefnames
\nocitenames

\def\today{\ifcase\month\or
  January\or February\or March\or April\or May\or June\or
  July\or August\or September\or October\or November\or December\fi
  \space\number\day, \number\year}

\DeclareMathOperator{\supp}{\mathrm{supp}}

\newtheorem{theorem}{Theorem}

\newtheorem{lemma}[theorem]{Lemma}
\newtheorem{proposition}[theorem]{Proposition}
\newtheorem{corollary}[theorem]{Corollary}

\renewcommand{\P}{\mathcal{P}}

\newcommand{\z}{\mathbb{Z}}

\renewcommand{\r}{\mathbb{R}}
\newcommand{\cp}{\mathbb{C}} 

\newcommand{\ft}{\widehat}

\newcommand{\bo}{\boldsymbol}

\newcommand{\ba}{{\bo a}}

\newcommand{\la}{\lambda}
\newcommand{\ga}{\gamma}
\newcommand{\al}{\alpha}
\newcommand{\be}{\beta}
\newcommand{\ep}{\varepsilon}

\newcommand{\p}{\varphi}

\newcommand{\bd}{{\tfrac12}B_d}

\renewcommand{\d}{\mathrm{d}}
\newcommand{\ov}[1]{\overline{#1}}

\newcommand{\cd}{\cdot}
\newcommand{\1}{{\bf 1}}
\newcommand{\ra}{\rightarrow}

\newcommand{\del}{{\bo \delta}}

\renewcommand{\Re}{{\rm Re}\,}

\renewcommand{\bd}{{\bo d}}

\begin{document}
\raggedbottom


\title[]{The Riemann zeta function at the zeros of real-rooted trigonometric polynomials}
\author{Felipe Gonçalves, Bryce Kerr, Oleksiy Klurman and Guilherme Vedana}
\date{\today}

\address{IMPA - Instituto de Matemática Pura e Aplicada, Rio de Janeiro, 22460-320, Brazil.}
\email{goncalves@impa.br}

\address{School of Science, University of New South Wales, Canberra, Australia}
\email{bryce.kerr@unsw.edu.au}

\address{School of Mathematics, University of Bristol, United Kingdom}
\email{lklurman@gmail.com}

\address{IMPA - Instituto de Matemática Pura e Aplicada, Rio de Janeiro, 22460-320, Brazil.}
\email{guilhermeisraelvedana@gmail.com}
\allowdisplaybreaks


\begin{abstract}
Let $P$ be a trigonometric polynomial with only real zeros. We prove that at most $\frac{2}{3}$ of
its zeros coincide with the zeros of the Riemann zeta function $\zeta(s)$ on the
critical line. We also show that
\begin{equation*}
    \sup_{\substack{T\leq x\leq2T\\P(x)=0}}
 \log \left|\zeta\!\left(\tfrac12+ix\right)\right|
 \geq
 (1+o(1))\sqrt{\frac{\log T}{6\log\log T}}.
\end{equation*}
These results extend the corresponding theorems of Li and Radziwi\l\l\ on the zeros of $\zeta(s)$ along vertical arithmetic progressions. The main novelty in our approach is the use of the work of Kurasov and Sarnak and Lee--Yang representations of Alon--Cohen-Vinzant, which show that the corresponding measure $\sum_{P(\gamma)=0} {\bf \delta}_\gamma$ is crystalline. 
\end{abstract}


\maketitle


\section{Introduction}

The behavior of the Riemann zeta function on discrete subsets of the critical line has been extensively studied from several points of view. Strengthening earlier nonvanishing results of Martin and Ng \cite{MN}, Li and Radziwi\l\l\ \cite{LR} proved that, for fixed $a>0$ and $b\in\r$,
\begin{equation}\label{RLresult}
 \frac{
  \#\left\{n\in\z\cap[N,2N]:
  \zeta\!\left(\frac12+i(an+b)\right)=0\right\}
 }{N}
 \leq \frac23+o(1)
 \qquad (N\to\infty),
\end{equation}
and this result holds unconditionally. Discrete first moments and related value-distribution questions had also been considered by Steuding and Wegert \cite{SW}. In a different direction, van Frankenhuijsen \cite{vF} obtained an explicit bound for the length of a finite arithmetic progression consisting entirely of zeta-zero ordinates.

Li and Radziwi\l\l\ also established the large-value estimate
\begin{equation}\label{eq:LR-large-values}
 \max_{N \leq n \leq 2N}
 \left|\zeta\!\left(\frac12+i(an+b)\right)\right|
 \geq
 \exp\!\left((1+o(1))
 \sqrt{\frac{\log N}{6\log\log N}}\right).
\end{equation}
Their argument is a discrete version of Soundararajan's resonance method \cite{Sound}. A triple log improvement of this bound via the resonance method and gcd-sums has subsequently been developed in many directions (see \cite{BDHS,BS,BS2,MS} for some of the results).
\subsection{From arithmetic progressions to trigonometric polynomials}

Our starting point is that an arithmetic progression is itself the zero set of a trigonometric
polynomial, since $\sin\!\left(\frac{\pi(x-b)}a\right)$ vanishes precisely at the points $b+an$, $n\in\z$.  This simple observation leads to the following  question:

\begin{quote}
\emph{Given a trigonometric polynomial $P$, how often can a real zero
of $P$ be the ordinate of a zero of the Riemann zeta function on the
critical line?}
\end{quote}

One naturally expects only finitely many such coincidences, since the zeros of trigonometric polynomials have completely different densities and statistics from those of the Riemann zeta function and are predictable. Nevertheless, even this statistical question seems unexplored in the literature.

More specifically, by a trigonometric polynomial we mean a function
\[
 P(z)=\sum_{j=1}^{J}a_j e^{2\pi i\lambda_jz} \qquad (a_j\in\cp,\quad \lambda_j\in\r).
\]
We say that \(P\) is \emph{real-rooted} when every zero of \(P\) is real. Throughout the paper, \(P\) is fixed, the multiset of its zeros is denoted by $\{\ga\}$, with each zero repeated according to its multiplicity. We briefly mention a few well-known and easy to derive facts: the number of such zeros in $[T,2T]$ is $cT+O(1).$ It is also well-known that the zeros split into two parts: a finite union of periodic sets and an aperiodic set, where the aperiodic part has finite intersection with any arithmetic progression. For more information about the gap distribution of zeros of real-rooted trigonometric polynomials see \cite{AV}. Our results provide non-trivial information when the trigonometric polynomial in question is \emph{not} a product of periodic trigonometric polynomials, that is, when the aperiodic part is present, since otherwise the statement would follow from the previous workof Li and Radziwi\l\l~\cite{LR} on arithmetic progressions. 

Our first main result is the following.

\begin{theorem}\label{thm:main}
Let \(P\) be a real-rooted trigonometric polynomial having at least one
zero. Then
\begin{equation*}
    \limsup_{T\to\infty} \frac{\#\{\ga \in [T,2T] : \zeta(\frac12 + i\ga)=P(\ga)=0\}}{\#\{\ga \in [T,2T] : P(\ga)=0\}} \leq \frac{2}{3},
\end{equation*}
where both sets are regarded as multisets.
\end{theorem}

Our second result extends the large-value estimate \eqref{eq:LR-large-values}.

\begin{theorem}\label{thm:main2}
Let \(P\) be a real-rooted trigonometric polynomial having infinitely many zeros. Then
\[
\sup_{\substack{T\leq\gamma\leq2T\\P(\gamma)=0}}
 \left|\zeta\!\left(\tfrac12+i\gamma\right)\right|
 \geq
 \exp\!\left((1+o(1))
 \sqrt{\frac{\log T}{6\log\log T}}\right).
\]
\end{theorem}

Real-rootedness is crucial for several reasons. On the technical side, one cannot count complex zeros of trigonometric polynomials using nonnegative smoothed sums without working with the modulus of the test function. This makes the application of a suitable summation formula much more complicated. Moreover, the test function in such a formula could not be compactly supported, because it would have to be evaluated at complex points, and hence we could not localize it. Nevertheless, we believe that the conclusions of our results hold for \emph{any} reasonable trigonometric polynomial, whether real-rooted or not.

The novelty of our results lies in the technique, in which real-rootedness plays a crucial role. Kurasov and Sarnak \cite{KS} showed that the counting measure of the zeros of a real-rooted trigonometric polynomial has discrete spectrum and its distributional Fourier transform is a sum of Dirac delta measures whose locations and weights can be computed explicitly.

To the best of our knowledge, this is the first use of such crystalline measures to prove results in analytic number theory and we hope this could inspire further related investigations.

We remark in passing that in view of the interest in the spectral interpretation of the zeros of $\zeta(s),$ one can immediately deduce the following corollary (see \cite{BK} for the background on quantum graphs).

\begin{corollary}\label{cor:quantum-graphs}
Let \(G\) be a finite compact metric graph equipped with the standard Kirchhoff--Neumann Laplacian. Then
\[
 \limsup_{T\to\infty}
 \frac{
  \#\left\{x\in\mathcal K_G\cap[T,2T]:
       \zeta\!\left(\frac12+ix\right)=0\right\}
 }{
  \#\bigl(\mathcal K_G\cap[T,2T]\bigr)
 }
 \leq\frac23,
\]
where \(\mathcal K_G\) is a multiset and all spectral values are counted with multiplicity.
\end{corollary}

\subsection{Remarks about the proofs}

The proof of Theorem~\ref{thm:main} follows mollification strategy of Li and Radziwi\l\l\ \cite{LR}. We replace the classical Poisson summation formula for an arithmetic progression with the Kurasov--Sarnak Fourier summation formula for the zero-counting measure of \(P\). This allows us to compare the discrete mollified first and second moments with their continuous counterparts. This comparison is the content of Proposition~\ref{prop:moments}, which may be of independent interest. We then bound the continuous second moment using the theorem of Balasubramanian, Conrey and Heath-Brown \cite{BCH}. A standard Cauchy--Schwarz argument then shows that at least one-third of the zeros of \(P\) in a dyadic interval are not zeta-zero ordinates.

For Theorem~\ref{thm:main2}, we adapt the resonator used by Li and Radziwi\l\l\ \cite[Section 5]{LR}. The significant additional difficulty is that the nonzero frequencies in the crystalline Fourier spectrum of the zeros of \(P\) may lie very close to logarithms of rational numbers. To deal with this issue, we separate the small crystalline frequencies and, crucially, remove a ``thin" collection of exceptional primes from the resonator (we further impose an auxiliary co-primality condition to deal with other diophantine constraints). The polynomial growth of the crystalline spectral measure coefficients controls the large frequencies, which we shall be able to dispense at an acceptable cost.

\section{Preliminaries}

\subsection{Crystalline measures and Lee--Yang polynomials}
\label{qcrsection}
We recall that a \textit{crystalline measure} on $\r$ is a pure point measure of the form $\mu=\sum_{\ga\in\Gamma} a_\ga\del_\ga$ that defines a tempered distribution and whose Fourier transform has the same form $\ft{\mu}=\sum_{\la\in\Lambda} b_\la\del_\la$, where $\Gamma,\Lambda\subset\r$ are discrete subsets and $\del_\ga$ is the delta mass at $\ga$. We call $\mu$ a \textit{Fourier quasicrystal} (FQ) when both $|\mu|$ and $|\ft{\mu}|$ are also tempered distributions. Kurasov and Sarnak \cite{KS} constructed a rich family of positive Fourier quasicrystals from Lee--Yang polynomials. Olevskii and Ulanovskii \cite{OU} proved that Fourier quasicrystals with unit masses arise from real-rooted exponential polynomials, while Alon, Cohen and Vinzant \cite{ACV} proved that every real-rooted exponential polynomial is, up to a zero-free exponential factor, a restriction of a Lee--Yang polynomial. Further structural and gap-distribution results appear in \cite{AV}, while broader classifications of Fourier summation formulas are developed in \cite{Go,GV,V26}.

A polynomial \(L\in\cp[z_1,\ldots,z_M]\) is called \emph{Lee--Yang} if it has no zeros when either
\[
 |z_1|,\ldots,|z_M|<1
 \qquad\text{or}\qquad
 |z_1|,\ldots,|z_M|>1.
\]
In particular, \(L(0)\neq0\). If \(\bo\ell =(\ell_1,\ldots,\ell_M)\in\r_+^{M}\), then $P(x):=L\left(e^{2\pi i\ell_1x},\ldots,e^{2\pi i\ell_Mx}\right)$ is a real-rooted trigonometric polynomial. Conversely, by \cite[Theorem~1.1]{ACV}, every real-rooted trigonometric polynomial admits a representation
\begin{equation}\label{eq:LY-representation}
 P(x)=C e^{2\pi i\lambda_0x} L\!\left( e^{2\pi i\ell_1x},\ldots, e^{2\pi i\ell_Mx}\right),
\end{equation}
where \(C\neq0\), \(\lambda_0\in\r\), \(L\) is Lee--Yang, and the positive numbers \(\ell_1,\ldots,\ell_M\) may be chosen linearly independent over \(\mathbb Q\). 

The following form of the Kurasov--Sarnak summation formula is the
bridge between the zero set of \(P\) and the analytic arguments in this
paper. We use the Fourier-transform normalization
\[
 \ft{f}(\xi)
 :=\int_{\r} f(x) e^{-2\pi ix\xi}\,\mathrm dx.
\]

\begin{theorem}[Kurasov--Sarnak; Alon--Cohen--Vinzant]\label{thm:KS}
Let \(P\) have the representation \eqref{eq:LY-representation}, and let $\bd=(d_1,\ldots,d_M)$ be the degrees of $L$ in each variable $z_j$. Then, for every Schwartz function $\p$, we have
\begin{equation*}
\sum_{\ga,P(\ga)=0} \p(\ga)=c_0\ft{\p}(0)+\sum_{\pm\ba\in\z_+^M} c_{\ba}\ft{\p}(\bo\ell\cdot\ba),
\end{equation*}
where $c_0=\bo\ell  \cdot \bd$, and for $\ba\in\z_+^M$
$$
\ov{c_\ba} = {c_{-\ba}}= -(\ba\cdot \bo\ell) [\bo z^\ba]\log{\left(\frac{L(\bo z)}{L(\bo 0)}\right)}.
$$

\end{theorem}
Here, $[\bo z^\ba]\log\!\left(L(\bo z)/L(\bo 0)\right)$ denotes the coefficient of $\bo z^\ba$ in the power-series expansion of $\log\!\left(L(\bo z)/L(\bo 0)\right)$ at $\bo 0$, and the sum $\sum_{P(\ga)=0}\p(\ga)$ counts the zeros with multiplicity. We also write $\z_+^M:=\{\ba=(a_1,\ldots,a_M):a_j\in\z,\, a_j\geq0,\ \text{and}\, \ba\neq(0,\ldots,0)\}$ and $\r_+:=\{x\in\r:x>0\}$. Finally, it is shown in \cite{KS} that $c_\ba=O_P(|\ba|)$ and hence, for any $A\geq1$, we have
$$
\sum_{\ba\in \z_+^M,  |\ba \cdot \bo\ell| \leq A} |c_\ba| = O_P(A^{M+1}).
$$
By \cite[Theorem 1.9]{AV}, we also have $\#\{\ga\in[T,2T]; P(\ga)=0\}=c_0T(1+o_P(1))$.

\subsection{Discrete first and second mollified moments}

We first record some preliminaries on the Riemann zeta function. As usual, we let $\zeta(s)$ denote the Riemann zeta function
\begin{align*}
&\zeta(s)=\sum_{n\geq1} \frac{1}{n^s} \quad  (\Re(s)>1).
\end{align*}
For\footnote{There is nothing special about the lower bound $1/10$. Any positive number $<1/2$ would work.} $1/10<\theta<1/2$ consider the mollifier
\begin{align*}
&M_{\theta}(s):=\sum_{n\leq T^\theta} \frac{b(n)}{n^s},
\end{align*}
where
\begin{align*}
b(m):=\mu(m)\left(1-\frac{\log m}{\log T^\theta}\right)
\quad\text{for }m\leq T^\theta,
\qquad b(m):=0\quad\text{otherwise}.
\end{align*}
Our argument is based on estimates for the discrete mollified first and second moments. Given any $\p\in C^\infty_c(\r)$ satisfying \begin{align}
    \label{eq:phi-properties}
    0\leq\p\leq1,\qquad \supp(\p)\subset[1,2],
\end{align} 
we set
\begin{align*}
I:=\sum_{\ga;P(\ga)=0} \zeta\cdot M_\theta\left(\frac{1}{2}+i\ga\right)\p\left(\frac{\ga}{T}\right) \ \ \text{and} \ \ 
J:=\sum_{\ga;P(\ga)=0} \left|\zeta\cdot M_\theta\left(\frac{1}{2}+i\ga\right)\right|^2\p\left(\frac{\ga}{T}\right).
\end{align*}
Regarding $\Lambda:=\{\ga;P(\ga)=0\}$ as a multiset, we define the proportion of its elements that are not ordinates of zeros of $\zeta$ on the critical line by
\begin{align*}
G(T):=\frac{\#\{\ga\in[T,2T]\cap\Lambda;\zeta\left(\frac{1}{2}+i\ga\right)\neq0\}}{\#\{\ga\in[T,2T]\cap\Lambda\}}.
\end{align*}
Theorem~\ref{thm:main} would then follow from the estimate
\begin{align*}
G(T)\geq \frac{1}{3}-o(1).
\end{align*}
From now on, we fix a real-rooted trigonometric polynomial $P$ and use the summation formula of Theorem \ref{thm:KS}, and the notation introduced there. We follow initially the strategy laid out in \cite{LR}. The Cauchy--Schwarz inequality implies
\begin{align}
\label{eq:IJ-cauchy}
|I|^2\leq (1+o(1))J\cd G(T)\cd c_0T.
\end{align}
Thus, the proof of Theorem \ref{thm:main} reduces to estimates for the mollified first and second moments, which are supplied by the proposition below.

\begin{proposition}[The mollified first and second moments]
\label{prop:moments}
Let $1/10<\theta<1/2$ and $\p\in C^\infty_c(\r)$ with $0\leq\p\leq1$ and $\supp(\p)\subset[1,2]$. Then, as $T\ra\infty$, we have
\begin{align*}
&I=\ft{\p}(0)c_0T+o_{\p,\theta}(T), \quad \text{ and }\\
&J=c_0\int_{\r} \left|\zeta\cd M_\theta\left(\frac{1}{2}+it\right)\right|^2\p\left(\frac{t}{T}\right)\d t+o_{\p,\theta}(T),
\end{align*}
where the implicit constant in the $o_{\p,\theta}(T)$ terms may depend on $\p$ and $\theta$.
\end{proposition}

We use this proposition in conjunction with an estimate of Balasubramanian, Conrey, and Heath-Brown~\cite{BCH} for handling $J$. Their result states that, if $\theta<9/17$, then
$$\int_{T}^{2T} \left|\zeta\cd M_\theta\left(\frac{1}{2}+it\right)\right|^2\d t\le \left(1+\frac{1}{\theta}+o_\theta(1)\right)T,$$
which, together with the assumption $\p(x)\le 1$, implies that $J\le c_0\left(1+\frac{1}{\theta}+o_\theta(1)\right)T$. Combining this estimate with~\eqref{eq:IJ-cauchy} gives
\begin{align*}
    G(T)\geq\frac{\ft{\p}(0)^2}{1+1/\theta+o_{\p,\theta}(1)}=\ft{\p}(0)^2\frac{\theta}{1+\theta}+o_{\p,\theta}(1).
\end{align*}
Now, for fixed $\p$ and $\theta<1/2$, we obtain
\begin{align*}
    \liminf_{T\ra\infty} G(T)\geq \ft{\p}(0)^2\frac{\theta}{1+\theta}.
\end{align*}
We then let $\theta\uparrow 1/2$ and take a sequence of admissible functions $\p$ such that $\ft{\p}(0)\ra1$ to obtain
\begin{align*}
    \liminf_{T\ra\infty} G(T)\geq \frac{1}{3}.
\end{align*}
This establishes Theorem \ref{thm:main}. It remains to prove Proposition~\ref{prop:moments} and to this end 
we estimate the first moment $I$ in Section~\ref{sec:I} and the second moment $J$ in Section~\ref{sec:J}.

\section{The first moment}\label{sec:I}
We prove that $I=\ft{\p}(0)c_0T+o_{\p,\theta}(T)$. We begin with the following  standard approximation for the Riemann zeta function:
\begin{align*}
\zeta\left(\frac{1}{2}+it\right)=\sum_{n\leq T} \frac{1}{n^{1/2+it}}+O\left(\frac{1}{T^{1/2}}\right), \quad T\le t \le 2T.
\end{align*}
This is a consequence of~\cite[Theorem~4.11]{Titch}, and it gives
\begin{align*}
I&=\sum_{P(\ga)=0} \left(\sum_{n\leq T} \frac{1}{n^{1/2+i\ga}}+O\left(\frac{1}{T^{1/2}}\right)\right)\sum_{m\leq T^\theta} \frac{b(m)}{m^{1/2+i\ga}}\p\left(\frac{\ga}{T}\right)\\
&=\sum_{P(\ga)=0} \sum_{n\leq T} \frac{1}{n^{1/2+i\ga}} \sum_{m\leq T^\theta} \frac{b(m)}{m^{1/2+i\ga}}\p\left(\frac{\ga}{T}\right)+o(T)
\end{align*}
because $\theta<1/2$. Writing $s=mn$, applying Theorem~\ref{thm:KS} to $\psi_s(x)=s^{-ix}\p(x/T)$, and using $\ft{\psi_s}(\xi)=T\ft{\p}\left(T(\xi+\log{s}/2\pi)\right)$, we obtain
\begin{align}
\label{eq:IR}
I&=\sum_{m\leq T^\theta} b(m)\sum_{n\leq T} \frac{1}{\sqrt{mn}}\sum_{P(\ga)=0} \frac{1}{(mn)^{i\ga}}\p\left(\frac{\ga}{T}\right)+o(T) \nonumber \\
&=\sum_{s\leq T^{1+\theta}} \frac{1}{\sqrt{s}} \sum_{P(\ga)=0} \frac{1}{s^{i\ga}} \p\left(\frac{\ga}{T}\right)A_{\theta,T}(s)+o(T) \nonumber \\
&=\sum_{s\leq T^{1+\theta}} \frac{1}{\sqrt{s}} \left\{c_0 T\ft{\p}\left(T\frac{\log{s}}{2\pi}\right)+\sum_{\pm\ba\in\z_+^M} c_{\ba} T\ft{\p}\left(T\left(\bo\ell\cd\ba+\frac{\log{s}}{2\pi}\right)\right)\right\}A_{\theta,T}(s)+o(T) \nonumber \\
&=c_0T\ft{\p}(0)+R_1+R_{2+}+R_{2-}+o(T),
\end{align}
where
\begin{align*}
    R_1&:=\sum_{\pm\ba\in\z_+^M} c_\ba T\ft{\p}\left(T\bo\ell\cd\ba\right)+\sum_{1<s\leq T^{1+\theta}} \frac{1}{\sqrt{s}} c_{0} T\ft{\p}\left(T\frac{\log{s}}{2\pi}\right)A_{\theta,T}(s),\\
    R_{2\pm}&:=\sum_{1<s\leq T^{1+\theta}} \frac{1}{\sqrt{s}}\sum_{\pm\ba\in\z_+^M}  c_{\ba} T\ft{\p}\left(T\left(\bo\ell\cd\ba+\frac{\log{s}}{2\pi}\right)\right)A_{\theta,T}(s), \text{ and} \\
    A_{\theta,T}(s)&:=\sum_{\substack{m|s\\ m\leq T^\theta\\ s/m\leq T}} b(m).
\end{align*}
Here the $\pm$ sign refers to $\pm\ba$. Observe that\footnote{In this paper, $\tau(k)=\sum_{d\mid k} 1$ denotes the divisor function, which satisfies $\tau(k)\ll_\ep k^\ep$ for any $\ep>0$.} $|A_{\theta,T}(s)|\leq \tau(s)\ll_\eta s^\eta$ for any $\eta>0$, and that
\begin{align*}
    \sum_{\ba\in\z_+^M} \frac{|c_{\ba}|}{|\bo\ell\cd\ba|^k}=O(1)
\end{align*}
for $k>M+1$. Since $\p\in C_c^\infty(\r)$, it follows that $\ft{\p}(x)\ll_{\p,k} |x|^{-k}$ for $|x|\geq1$ and any $k\ge 1$. Hence, we obtain $R_1=o_{\p}(1)$ and $R_{2+}=o_{\p}(1)$.

\subsection*{Key step:}
Now we claim that $R_{2-}=o_{\p,\theta}(T)$. 
Indeed, the difficulty arises when $T\bo\ell\cd\ba\left(-1+\frac{\log{s}}{2\pi\bo\ell\cd\ba}\right)$ is close to $0$. To account for this case, we divide $R_{2-}$ into ``good'' and ``bad'' parts. For $0<\ep<1$, to be chosen below, let
\begin{align*}
G:=&\left\{(\ba,s);\ba\in\z_+^M, 1<s\leq T^{1+\theta},\left|1-\frac{\log{s}}{2\pi\bo\ell\cd\ba}\right|>\frac{1}{T^{1-\ep}}\right\}, \text{ and }\\
B:=&\left\{(\ba,s);\ba\in\z_+^M, 1<s\leq T^{1+\theta},\left|1-\frac{\log{s}}{2\pi\bo\ell\cd\ba}\right|\leq\frac{1}{T^{1-\ep}}\right\}
\end{align*}
and let $R_G$ and $R_B$ denote the sums $R_{2-}$ restricted to $G$ and $B$, respectively. Using $|A_{\theta,T}(s)|\ll_\ep s^\ep$ and $\ft{\p}(x)\ll_{\p,\ep} |x|^{-k}$ for $k=k(\ep)=\left\lceil 3(\frac{1}{2\ep}+1)\right\rceil$, which we may assume satisfies $k>M+1$ by taking a smaller $\ep$ if necessary, we conclude that $R_G=o_{\p,\theta,\ep}(T)$.

To estimate $R_B$, we consider the cases $s\leq T$ and $s>T$ separately. If $s\leq T$, then $A_{\theta, T}(s)=\sum_{m|s} b(m)$, and since $b(m)=\mu(m)+O\left(\frac{\log{m}}{\log{T}}\right)$, we may use $\sum_{m|s} \mu(m)=0$ to deduce that $|A_{\theta,T}(s)|\ll_{\ep} s^\ep/\log{T}$. Denote by $R_{\leq T}$ the sum $R_B$ restricted to $s\leq T$. Then
\begin{align*}
    R_{\leq T}\ll_{\ep} \frac{T}{\log{T}} \sum_{\ba\in\z_+^M} |c_{\ba}| \sum_{\substack{1<s\leq T\\ s\in B_\ba}} s^{-1/2+\ep},
\end{align*}
where $B_{\ba}:=\{s;(\ba,s)\in B\}$. For a fixed $\ba\in\z_+^M$, we now derive a lower bound for $s$, an upper bound for $\ba$, and an estimate for $\# B_\ba$. If $s\in B_\ba$, then
\begin{equation}
\label{eq:size_set_B_a}
2\pi\bo\ell\cd\ba\left(1-\frac{1}{T^{1-\ep}}\right)\leq\log{s}\leq2\pi\bo\ell\cd\ba\left(1+\frac{1}{T^{1-\ep}}\right).
\end{equation}  
For $T$ sufficiently large, we have $(1+\ep)^{-1}\leq 1-T^{\ep-1}$, from which we obtain $2\pi\bo\ell\cd\ba\leq(1+\ep)\log{s}\leq (1+\varepsilon)\log{T}$ (otherwise $B_{\ba}=\emptyset$), and $e^{2\pi\bo\ell\cd\ba}\leq s^{1+\varepsilon}$. Moreover, by the mean value theorem,
\begin{align*}
\#B_{\ba}&\leq \exp{\left(2\pi\bo\ell\cd\ba\left(1+\frac{1}{T^{1-\ep}}\right)\right)}-\exp{\left(2\pi\bo\ell\cd\ba\left(1-\frac{1}{T^{1-\ep}}\right)\right)}+1\\
&\leq 4\pi\bo\ell\cd\ba e^{2\pi(1+\varepsilon)\bo\ell\cd\ba}\frac{1}{T^{1-\varepsilon}}+1.
\end{align*}
Combining these estimates with the fact that $|c_\ba|=O(|\ba|)$, we obtain
\begin{align*}
    R_{\leq T}&\ll_\ep \frac{T}{\log{T}} \sum_{\substack{\ba\in\z_+^M\\ 2\pi\bo\ell\cd\ba\leq(1+\ep)\log{T}}} |c_\ba| \sum_{\substack{1<s\leq T\\ s\in B_\ba}} s^{-1/2+\ep}\\
    &\ll_\ep \frac{T}{\log{T}} \sum_{\substack{\ba\in\z_+^M \\ 2\pi\bo\ell\cd\ba\leq (1+\varepsilon)\log{T} }} |\ba|\#B_{\ba} \exp\left(\frac{2\pi\bo\ell\cd\ba}{1+\ep}\left(-\frac{1}{2}+\ep\right)\right)\\
    &\ll_\ep \frac{T}{\log{T}}R_{\leq T,1}+ \frac{T^{\varepsilon}}{\log{T}}R_{\leq T,2},
\end{align*}
where
\begin{align*}
    R_{\leq T,1}=\sum_{\substack{\ba\in\z_+^M \\ 2\pi\bo\ell\cd\ba\leq (1+\varepsilon)\log{T} }} |\ba| \exp\left(\frac{2\pi\bo\ell\cd\ba}{1+\ep}\left(-\frac{1}{2}+\ep\right)\right)=O(1)
\end{align*}
provided that we take $\ep<1/4$, and
\begin{align*}
R_{\leq T,2}&=\sum_{\substack{\ba\in\z_+^M \\ 2\pi\bo\ell\cd\ba\leq (1+\varepsilon)\log{T} }} |\ba| 4\pi\bo\ell\cd\ba\exp\left[2\pi\bo\ell\cd\ba\left(1+\ep+\frac{1}{1+\ep}\left(-\frac{1}{2}+\ep\right)\right) \right]\\
&\ll T^{(1+\ep)^2-1/2+\ep}   \log{T}  \sum_{\substack{\ba\in\z_+^M \\ 2\pi\bo\ell\cd\ba\leq (1+\varepsilon)\log{T} }} |\ba|\\
&\ll  T^{(1+\ep)^2-1/2+\ep}   (\log{T})^{M+2} \ll_\ep T^{1/2+\ep(M+6)}.
\end{align*}
Taking a smaller $\ep>0$, if necessary, so that $1/2+\ep(M+6)+\ep<3/4$, we obtain $T^\ep R_{\leq T,2}/\log{T}\ll_\ep T^{3/4}/\log{T}$. Therefore, $R_{\leq T}=o_\ep(T)$.

It remains to estimate the sum
\begin{align*}
    R_{>T}&:=T\sum_{T<s\leq T^{1+\theta}} s^{-1/2}\sum_{\substack{\ba\in\z_+^M\\ (\ba,s)\in B}}  \ov{c_{\ba}} \ft{\p}\left(T\bo\ell\cd\ba\left(-1+\frac{\log{s}}{2\pi\bo\ell\cd\ba}\right)\right)A_{\theta,T}(s).
\end{align*}
From \eqref{eq:size_set_B_a}, it follows that, for $T$ sufficiently large, $2\pi\bo\ell\cd\ba \leq (1+\theta+2\ep) \log{T}$. Using the trivial estimate $|A_{\theta,T}(s)|\ll_\ep s^\ep$ and $s^{-1/2+\ep}<T^{-1/2+\ep}$, we obtain
\begin{align*}
    R_{>T}&\ll_\ep T \sum_{\substack{\ba\in\z_+^M\\ 2\pi\bo\ell\cd\ba\leq (1+\theta+2\ep)\log{T}}} |\ba| \sum_{\substack{T<s\leq T^{1+\theta}\\ s\in B_\ba}} s^{-1/2+\ep}\\
    &\ll_\ep T \sum_{\substack{\ba\in\z_+^M\\ 2\pi\bo\ell\cd\ba\leq (1+\theta+2\ep)\log{T}}} |\ba|  T^{-1/2+\ep}\# B_\ba\\
    &\ll T^{1/2+\ep} R_{>T,1}+T^{-1/2+2\ep} R_{>T,2},
\end{align*}
where
\begin{align*}
    R_{>T,1}&=\sum_{\substack{\ba\in\z_+^M\\ 2\pi\bo\ell\cd\ba\leq (1+\theta+2\ep)\log{T}}} |\ba| \ll (\log{T})^{M+1}, \text{ and }\\
    R_{>T,2}&=\sum_{\substack{\ba\in\z_+^M\\ 2\pi\bo\ell\cd\ba\leq (1+\theta+2\ep)\log{T}}} |\ba| 4\pi\bo\ell\cd\ba e^{2\pi\bo\ell\cd\ba(1+\ep)}\ll T^{(1+\ep)(1+\theta+\ep)} \log{T} \sum_{\substack{\ba\in\z_+^M\\ 2\pi\bo\ell\cd\ba\leq (1+\theta+2\ep)\log{T}}} |\ba|\\
    &\ll T^{1+\theta+6\ep} (\log{T})^{M+2},
\end{align*}
hence $R_{>T}\ll_\ep T^{1/2+\theta+\ep(M+9)}=o_\ep(T)$ if we choose $\ep$ so that $\ep(M+9)<1/2-\theta$. Therefore, for fixed $1/10<\theta<1/2$, we choose $\ep>0$ to satisfy all the requirements above, which depend only on $\theta$ and the fixed polynomial $P$, and obtain $I=\ft{\p}(0) c_0 T+o_{\p,\theta}(T)$, as desired.

\section{The second moment}
\label{sec:J}
We next prove the second estimate of Proposition~\ref{prop:moments}, namely,
$$J=c_0\int_{\r} \left|\zeta\cd M_\theta\left(\frac{1}{2}+it\right)\right|^2\p\left(\frac{t}{T}\right)\d t+o_{\p,\theta}(T).$$
Here we recall that
$$J=\sum_{P(\ga)=0} \left|\zeta\cdot M_\theta\left(\frac{1}{2}+i\ga\right)\right|^2\p\left(\frac{\ga}{T}\right).$$
We will need a finer estimate for $|\zeta(1/2+it)|$ than the one used for the first moment. Let $G$ be an even, real entire function such that $G(0)=1$ and that has rapid decay along vertical strips: $|G(x+iy)|\ll_{\al,A}(1+|y|)^{-A}$ uniformly in the range $-\al<x<\al$, for any $0<\al<\infty$ and $A>0$. For $\ep>0$, define
\begin{equation}
\label{eq:W_def}
W(x):=\frac{1}{2\pi i} \int_{(\ep)} x^{-w}G(w)\frac{\d w}{w},\qquad x>0.
\end{equation}
The function $W$ has the following properties: it is real-valued, $W\in C^{\infty}(0,\infty)$, and for every $A>0$ it satisfies $W(x)=1+O_A(x^A)$ as $x\to0^+$ and $W(x)\ll_A x^{-A}$ as $x\ra\infty$.

We now use the following approximate functional equation for the Riemann zeta function on the critical line; see \cite[Lemma 4]{LR}.

\begin{lemma}[Approximate functional equation for the Riemann zeta function]
\label{lemma:approx_riem_zeta}
For $T<t<2T$, $\ep>0$, and $W$ as above, we have
\begin{equation*}
\left|\zeta\left(\frac{1}{2}+it\right)\right|^2=2\sum_{mn<T^{1+\ep}} \frac{1}{\sqrt{mn}}\cdot\left(\frac{m}{n}\right)^{it}W\left(\frac{2\pi mn}{t}\right)+O(T^{-2/3}).
\end{equation*}
\end{lemma}
We are now ready to estimate $J$.
\subsection*{Step 1: Decomposition of the second moment via the Kurasov--Sarnak summation formula}
\begin{align*}
J&=\sum_{P(\ga)=0} \left|\zeta\cdot M_\theta\left(\frac{1}{2}+i\ga\right)\right|^2\p\left(\frac{\ga}{T}\right)\\
&=2\sum_{P(\ga)=0} \sum_{mn<T^{1+\ep}} \sum_{h,k\leq T^{\theta}}\frac{b(h)b(k)}{\sqrt{mnhk}}W\left(\frac{2\pi mn}{\ga}\right) \left(\frac{mk}{nh}\right)^{i\ga}\p\left(\frac{\ga}{T}\right)\\
&+O(T^{-2/3})\sum_{P(\ga)=0} \sum_{h,k\leq T^{\theta}} \frac{b(h)b(k)}{\sqrt{hk}}\left(\frac{k}{h}\right)^{i\ga}\p\left(\frac{\ga}{T}\right)=J_1+J_2.
\end{align*}
We have $J_2\ll T^{5/6}$ by direct computation. For $J_1$, we apply the Kurasov--Sarnak summation formula, Theorem~\ref{thm:KS}, with
\begin{align}
\label{eq:fdef}
    f(x)=f_{m,n,T}(x):=W(2\pi mn/x)\p(x/T),\qquad x>0,
\end{align}
and set $f(x)=0$ if $x\leq0$. Note that $f\in C^\infty_c(\r)$ because $W\in C^{\infty}(0,\infty)$ and $\p(\,\cdot\,/T)$ vanishes outside the interval $[T,2T]$. We obtain
\begin{align*}
J_1&=2\sum_{mn<T^{1+\ep}} \frac{1}{\sqrt{mn}}\sum_{h,k\leq T^{\theta}} \frac{b(h)b(k)}{\sqrt{hk}}\left(\sum_{P(\ga)=0}\left(\frac{mk}{nh}\right)^{i\ga}f(\ga)\right)\\
&=2\sum_{mn<T^{1+\ep}} \frac{1}{\sqrt{mn}}\sum_{h,k\leq T^{\theta}} \frac{b(h)b(k)}{\sqrt{hk}}\left(c_0\ft{f}\left(-\frac{\log{\frac{mk}{nh}}}{2\pi}\right)+\sum_{\pm\ba\in\z_+^M} c_{\ba}\ft{f}\left(\bo\ell\cd\ba-\frac{\log{\frac{mk}{nh}}}{2\pi}\right)\right).
\end{align*}

\subsection*{Step 2: The main term \texorpdfstring{$\ba=0$}{a = 0}}
After isolating the main term with $\ba=0$, we write it as an integral:
\begin{align*}
   &\sum_{mn<T^{1+\ep}} \frac{2}{\sqrt{mn}}\sum_{h,k\leq T^{\theta}} \frac{b(h)b(k)}{\sqrt{hk}}c_0\ft{f}\left(-\frac{\log{\frac{mk}{nh}}}{2\pi}\right)\\
   &=2c_0\int_{-\infty}^{\infty}f(t)\sum_{mn<T^{1+\ep}}\frac{1}{\sqrt{mn}}\sum_{h,k\leq T^{\theta}} \frac{b(h)b(k)}{\sqrt{hk}}\left(\frac{mk}{nh}\right)^{it}\d t \\ 
   &=c_0\int_{\r} \left|\zeta\cd M_\theta\left(\frac{1}{2}+it\right)\right|^2\p\left(\frac{t}{T}\right)\d t+o(T)
\end{align*}
after recalling~\eqref{eq:fdef} and using Lemma~\ref{lemma:approx_riem_zeta}. Hence, we may write
\begin{align*}
J_1=c_0\int_{\r} \left|\zeta\cd M_\theta\left(\frac{1}{2}+it\right)\right|^2\p\left(\frac{t}{T}\right)\d t+J_3+o(T),
\end{align*}
where 
\begin{align*}
J_3
&=2\sum_{mn<T^{1+\ep}} \frac{1}{\sqrt{mn}}\sum_{h,k\leq T^{\theta}} \frac{b(h)b(k)}{\sqrt{hk}}\sum_{\pm\ba\in\z_+^M} c_{\ba}\ft{f}\left(\bo\ell\cd\ba-\frac{\log{\frac{mk}{nh}}}{2\pi}\right).
\end{align*}

\subsection*{Step 3: The terms with \texorpdfstring{$\ba\neq0$}{a is nonzero} are the error terms}
We next estimate $J_3$. Since $\ov{\ft{f}\left(\bo\ell\cd\ba-\frac{\log{\frac{mk}{nh}}}{2\pi}\right)}=\ft{f}\left(-\bo\ell\cd\ba-\frac{\log{\frac{nh}{mk}}}{2\pi}\right)$ and $c_{-\ba}=\overline{c_{\ba}}$, we have
\begin{align*}
J_3=4\Re\sum_{\ba\in\z_+^M} \sum_{\substack{mn<T^{1+\ep}\\ h,k\leq T^{\theta}}} \frac{b(h)b(k)}{\sqrt{mnhk}} c_{\ba}\ft{f}\left(\bo\ell\cd\ba-\frac{\log{\frac{mk}{nh}}}{2\pi}\right)=R_G+R_B,
\end{align*}
where we once again split the sum into ``good'' and ``bad'' parts as follows:
\begin{align*}
&G:=\left\{(\ba,m,n,h,k);\left|1-\frac{\log{\frac{mk}{nh}}}{2\pi\bo\ell\cd\ba}\right|>\frac{1}{T^{1-\eta}}\right\}\\
\text{and} \ \ &B:=\left\{(\ba,m,n,h,k);\left|1-\frac{\log{\frac{mk}{nh}}}{2\pi\bo\ell\cd\ba}\right|\leq\frac{1}{T^{1-\eta}}\right\},
\end{align*}
where $R_G$ and $R_B$ denote the sums $J_3$ restricted to the sets $G$ and $B$, respectively. Since $\ft{f}(x)\ll_A \frac{T}{(1+T|x|)^A}$, with the implicit constant in $\ll_A$ independent of $m,n$, it follows that $R_G=o_{\eta}(T)$ if we choose $A=\lceil 3/\eta \rceil$. It remains to bound $R_B$. Observe that $(\ba,m,n,h,k)\in B$ implies that $2\pi\bo\ell\cd\ba\leq 4\log{T}$ for sufficiently large $T$. Hence,
\begin{align*}
B\subset B_1:=\left\{(\ba,m,n,h,k);\left|2\pi\bo\ell\cd\ba-\log{\frac{mk}{nh}}\right|\leq\frac{1}{T^{1-2\eta}}=\frac{1}{T^{1-\tilde{\eta}}}\right\}.
\end{align*}
Using the inequality $|e^a-e^b|\leq e^{a+|b-a|}|b-a|$, valid for any $a,b\in\r$, we deduce that
\begin{align*}
B_1\subset B_2:=\left\{(\ba,m,n,h,k);\left|\frac{mk}{nh}-e^{2\pi\bo\ell\cd\ba}\right|\leq \frac{e^{2\pi\bo\ell\cd\ba}}{T^{1-\tilde{\eta}}}\right\}
\end{align*}
after redefining $\tilde{\eta}$ if necessary. Therefore, we may replace the summation defining $R_B$ by one over the larger set $B_2$. More precisely, after redefining $G$ and $B$, we may take $G_2:=G\backslash B_2$ and $B_2$ as the ``good'' and ``bad'' parts, respectively. We then still have $J_3=R_{G_2}+R_{B_2}$, with $R_{G_2}=o_{\eta}(T)$. For simplicity, we drop the subscript $2$ and write only $R_G$ and $R_B$. We are therefore left with the term
\begin{align*}
R_B=4\Re{\sum_{\ba\in\z_+^M} c_{\ba} \sum_{\substack{mn<T^{1+\ep}\\ h,k\leq T^{\theta}\\ (m,n,h,k)\in B_{\ba}}} \frac{b(h)b(k)}{\sqrt{mnhk}}\ft{f}\left(\bo\ell\cd\ba-\frac{\log{\frac{mk}{nh}}}{2\pi}\right)},
\end{align*}
where $B_{\ba}:=\{(m,n,h,k);(\ba,m,n,h,k)\in B_2\}$. Fix $\ba\in\z_+^M$ and partition the sum according to $(mk,nh)$. Let
$$C_{\ba}:=\left\{(r,s);r,s\geq1,\operatorname{gcd}(r,s)=1,\, \left|\frac{r}{s}-e^{2\pi\bo\ell\cd\ba}\right|\leq \frac{e^{2\pi\bo\ell\cd\ba}}{T^{1-\tilde{\eta}}}\right\}\subset \z\times\z.$$
With this notation,
\begin{align*}
R_B=4\Re{\sum_{\ba\in\z_+^M} c_{\ba}\sum_{(r,s)\in C_{\ba}} \sum_{d\geq1} \frac{1}{d} \sum_{\substack{mn<T^{1+\ep}\\ h,k\leq T^{\theta} \\ mk=rd \\ nh=sd}} \frac{b(h)b(k)}{\sqrt{rs}}\ft{f}\left(\bo\ell\cd\ba-\frac{\log{\frac{r}{s}}}{2\pi}\right)}.
\end{align*}
Define $K_{\ba}^2:=T^{1-\tilde{\eta}}e^{-2\pi\bo\ell\cd\ba}/4$. We partition the sum over $s$ at $K_{\ba}$ to obtain $R_B=R_{B,+}+R_{B,-}$, where
\begin{align*}
R_{B,+}=4\Re{\sum_{\ba\in\z_+^M} c_{\ba}\sum_{\substack{(r,s)\in C_{\ba} \\ s\ge K_{\ba} }} \sum_{d\geq1} \frac{1}{d} \sum_{\substack{mn<T^{1+\ep}\\ h,k\leq T^{\theta}\\ mk=rd \\ nh=sd}} \frac{b(h)b(k)}{\sqrt{rs}}\ft{f}\left(\bo\ell\cd\ba-\frac{\log{\frac{r}{s}}}{2\pi}\right)}   
\end{align*}
and 
\begin{align*}
R_{B,-}=4\Re{\sum_{\ba\in\z_+^M} c_{\ba}\sum_{\substack{(r,s)\in C_{\ba} \\ s< K_{\ba} }} \sum_{d\geq1} \frac{1}{d} \sum_{\substack{mn<T^{1+\ep}\\ h,k\leq T^{\theta}\\ mk=rd \\ nh=sd}} \frac{b(h)b(k)}{\sqrt{rs}}\ft{f}\left(\bo\ell\cd\ba-\frac{\log{\frac{r}{s}}}{2\pi}\right)}.    
\end{align*}
We first show that $R_{B,+}$ is small. Indeed, if $(r,s)\in C_{\ba}$, then $rs\geq s^2e^{2\pi\bo\ell\cd\ba}(1-T^{-1+\tilde{\eta}})\geq s^2e^{2\pi\bo\ell\cd\ba}/4$ for every sufficiently large $T$, depending on $\tilde{\eta}$. We also have $rs\leq mnhk< T^{1+\ep+2\theta}$, and hence $s\leq M_\ba$, where $M_\ba:=2 T^{1/2+\ep+\theta} e^{-\pi\bo\ell\cd\ba}$; moreover, $d\leq T^2$. If $(r,s)\in C_\ba$ has a contributing tuple of integers $m,n,h,k$ (so that the innermost constrained sum is nonempty), then the definition of $C_\ba$ also implies that $2\pi\bo\ell\cd\ba\leq 4 \log{T}$.

Next, we split the sum into dyadic blocks $2^{n-1}K_\ba\leq s<2^n K_\ba$ for $1\leq n\le N$, where $N:=\left\lceil\log{\frac{2M_{\ba}}{K_{\ba}}}/\log{2}\right\rceil$, so that $2^N K_\ba\geq 2M_{\ba}$. Since $\ft{f}(x)\ll T$, with an implicit constant independent of $m,n$, we obtain
\begin{align*}
R_{B,+}&\ll T\sum_{\substack{\ba\in\z_+^M\\ 2\pi\bo\ell\cd\ba\leq 4\log{T}}} |c_{\ba}| \sum_{n=1}^{N} \sum_{\substack{(r,s)\in C_{\ba}\\ 2^{n-1}K_\ba\leq s< 2^nK_\ba}} \frac{1}{\sqrt{rs}} \sum_{d\leq T^2} \frac{1}{d} \sum_{\substack{mn<T^{1+\ep}\\ h,k\leq T^\theta\\ mk=rd\\ nh=sd}} 1\\
&\ll_\ep T\sum_{\substack{\ba\in\z_+^M\\ 2\pi\bo\ell\cd\ba\leq 4\log{T}}} |c_{\ba}| \sum_{n=1}^{N} \sum_{\substack{(r,s)\in C_{\ba}\\ 2^{n-1}K_\ba\leq s< 2^nK_\ba}} (rs)^{-1/2+\ep} \sum_{d\leq T^2} d^{2\ep-1}
\end{align*}
because
\begin{align*}
&\sum_{\substack{mn<T^{1+\ep}\\ h,k\leq T^\theta\\ mk=rd\\ nh=sd}} 1\leq \tau(rd)\tau(sd)\ll_\ep (rs)^\ep d^{2\ep},
\end{align*}
where we recall that $\tau(\,\cdot\,)$ is the divisor function. Hence,
\begin{equation}
\label{eq:sum_S_2}
R_{B,+}\ll_\ep T^{1+4\ep} \sum_{\substack{\ba\in\z_+^M\\ 2\pi\bo\ell\cd\ba\leq 4\log{T}}} |c_{\ba}| \sum_{n=1}^{N} \sum_{\substack{2^{n-1}K_\ba\leq s< 2^nK_\ba\\ (r,s)\in C_{\ba}}} (rs)^{-1/2+\ep}.
\end{equation}
We now consider the inner summation over $s$. If $(r,s)\in C_\ba$, then $r\geq s e^{2\pi\bo\ell\cd\ba} (1-T^{-1+\tilde{\eta}})$. Combining this with $s\geq 2^{n-1} K_\ba$, we obtain $(rs)^{-1/2+\ep}\leq (2^{n-2} K_\ba e^{\pi\bo\ell\cd\ba})^{-1+2\ep}$. Moreover, the number of pairs $(r,s)\in C_{\ba}$ such that $2^{n-1} K_\ba\leq s< 2^n K_\ba$ is at most
\begin{align*}
2(2^n K_\ba)^2\frac{e^{2\pi\bo\ell\cd\ba}}{T^{1-\tilde{\eta}}}+1,
\end{align*}
because Farey fractions with denominator at most $2^n K_\ba$ are spaced at least $(2^n K_\ba)^{-2}$ and the condition $(r,s)\in C_{\ba}$ implies $r/s$ lies in an interval of length $2e^{2\pi \bo\ell\cd \ba}/T^{1-\tilde{\eta}}$. Therefore,
\begin{align*}
\sum_{\substack{(r,s)\in C_{\ba}\\ 2^{n-1}K_\ba\leq s< 2^nK_\ba}} (rs)^{-1/2+\ep}\leq (2^{n-2}K_\ba e^{\pi\bo\ell\cd\ba})^{-1+2\ep}\left(2(2^n K_\ba)^2\frac{e^{2\pi\bo\ell\cd\ba}}{T^{1-\tilde{\eta}}}+1\right).
\end{align*}
Substituting this estimate into~\eqref{eq:sum_S_2} gives
\begin{align}
\label{eq:RB+}
R_{B,+}\ll_\ep T^{1+4\ep} \sum_{\substack{\ba\in\z_+^M\\ 2\pi\bo\ell\cd\ba\leq 4\log{T}}} |c_{\ba}| \sum_{n=1}^{N} (2^{n-2}K_\ba e^{\pi\bo\ell\cd\ba})^{-1+2\ep}\left(2(2^n K_\ba)^2\frac{e^{2\pi\bo\ell\cd\ba}}{T^{1-\tilde{\eta}}}+1\right).
\end{align}
Using the estimate $\sum_{n=1}^{N} 2^n\ll 2^{N+1}\leq 8 M_\ba/K_\ba$ and taking $0<\tilde{\eta}\leq\ep$, the contribution of the $+1$ in \eqref{eq:RB+} is
\begin{align*}
    &\ll_\ep T^{1+4\ep} \sum_{\substack{\ba\in\z_+^M\\ 2\pi\bo\ell\cd\ba\leq 4\log{T}}} |c_{\ba}| \sum_{n=1}^{N} (2^{n-2}K_\ba e^{\pi\bo\ell\cd\ba})^{-1+2\ep}\ll T^{1/2+\theta+7\ep} \sum_{\substack{\ba\in\z_+^M\\ 2\pi\bo\ell\cd\ba\leq 4\log{T}}} |\ba|\\
    &\ll T^{1/2+\theta+7\ep} (\log{T})^{M+1} \ll_\ep T^{1/2+\theta+(M+8)\ep}=o_\ep(T)
\end{align*}
provided that $(M+8)\ep<1/2-\theta$. The other term in \eqref{eq:RB+} contributes
\begin{align*}
    &\ll_\ep T^{1+4\ep} \sum_{\substack{\ba\in\z_+^M\\ 2\pi\bo\ell\cd\ba\leq 4\log{T}}} |c_{\ba}| \sum_{n=1}^{N} (2^{n-2}K_\ba e^{\pi\bo\ell\cd\ba})^{-1+2\ep}2(2^n K_\ba)^2\frac{e^{2\pi\bo\ell\cd\ba}}{T^{1-\tilde{\eta}}}\\
    &\ll T^{1/2+\theta+10\ep} \sum_{\substack{\ba\in\z_+^M\\ 2\pi\bo\ell\cd\ba\leq 4\log{T}}} |c_{\ba}|\ll_\ep T^{1/2+\theta+(M+11)\ep}=o_\ep(T) 
\end{align*}
provided that $(M+11)\ep<1/2-\theta$. Therefore, taking $2(M+11)\ep=1/2-\theta$ yields $R_{B,+}=o_\theta(T)$.

It remains to estimate $R_{B,-}$. We claim that, for each $\ba$, there is at most one pair $(r_{\ba},s_{\ba})\in C_{\ba}$ with $s_{\ba}< K_{\ba}$. Indeed, if there were two such pairs, say $(r_{\ba},s_{\ba})$ and $(r'_{\ba},s'_{\ba})$, then we would have
$$4\frac{e^{2\pi\bo\ell\cd\ba}}{T^{1-\tilde{\eta}}}=\frac{1}{K_{\ba}^2} \leq \frac{1}{s_{\ba}s'_{\ba}}\le \left|\frac{r_{\ba}}{s_{\ba}}-\frac{r'_{\ba}}{s'_{\ba}}\right|\le 2\frac{e^{2\pi\bo\ell\cd\ba}}{T^{1-\tilde{\eta}}},$$
which is a contradiction. With $(r_{\ba},s_{\ba})$ as above, we have
\begin{align*}
R_{B,-}=4\Re{\sum_{\substack{\ba\in\z_+^M \\ s_{\ba}< K_{\ba}}} c_{\ba} \sum_{d\geq1} \frac{1}{d} \sum_{\substack{mn<T^{1+\ep}\\ h,k\leq T^{\theta}\\ mk=r_{\ba}d \\ nh=s_{\ba}d}} \frac{b(h)b(k)}{\sqrt{r_{\ba}s_{\ba}}}\ft{f}\left(\bo\ell\cd\ba-\frac{\log{\frac{r_{\ba}}{s_{\ba}}}}{2\pi}\right)}.
\end{align*}
Recalling~\eqref{eq:fdef}, we obtain
\begin{align*}
&R_{B,-}=4\Re \sum_{\substack{\ba\in\z_+^M\\ s_{\ba}< K_{\ba}}}\frac{c_{\ba}}{\sqrt{r_{\ba}s_{\ba}}}\int_{T}^{2T} \p\left(\frac{t}{T}\right)\exp\left({-2\pi it\left(\bo\ell\cd\ba-\frac{\log{r_{\ba}/s_{\ba}}}{2\pi}\right)}\right)F(r_{\ba},s_{\ba},t)\d t,\\
&\text{where } F(r,s,t):=\sum_{d\geq1} \frac{1}{d} \sum_{h,k\leq T^\theta} b(h)b(k) \sum_{\substack{mn<T^{1+\ep}\\ mk=rd\\ nh=sd}} W\left(\frac{2\pi mn}{t}\right).
\end{align*}
Since $T\leq t\leq 2T$ and $W(x)\ll_A x^{-A-1}$, the tail satisfies
\begin{align*}
    F_{tail}(r,s,t)&:=\sum_{d\geq1} \frac{1}{d} \sum_{h,k\leq T^\theta} b(h)b(k) \sum_{\substack{mn\geq T^{1+\ep}\\ mk=rd\\ nh=sd}} W\left(\frac{2\pi mn}{t}\right)\\
    &\ll_A \sum_{d\geq1} \frac{1}{d} \sum_{h,k\leq T^\theta} \sum_{\substack{mn\geq T^{1+\ep}\\ mk=rd\\ nh=sd}} \left|\frac{2\pi mn}{t}\right|^{-A-1}\ll T^{-A\ep} \sum_{d\geq1} \frac{1}{d} \sum_{h,k\leq T^\theta} \sum_{\substack{mn\geq T^{1+\ep}\\ mk=rd\\ nh=sd}} \frac{thk}{2\pi rdsd}\\
    &\ll T^{-A\ep+1+2\theta} \sum_{d\geq1} \frac{1}{d^3} \sum_{\substack{h,k\leq T^\theta\\ h\mid rd\\ k\mid sd}}1\ll T^{-A\ep+1+2\theta} \sum_{d\geq1} \frac{\tau(sd)\tau(rd)}{d^3}\ll T^{-B}
\end{align*}
for arbitrarily large $B>0$. Hence, we may redefine
\begin{align*}
   F(r,s,t):=\sum_{d\geq1} \frac{1}{d} \sum_{h,k\leq T^\theta} b(h)b(k) \sum_{\substack{m,n\geq1\\ mk=rd\\ nh=sd}} W\left(\frac{2\pi mn}{t}\right). 
\end{align*}
To control $F(r,s,t)$, we appeal to \cite[Lemma 1]{LR}. The statement of \cite[Lemma 1]{LR} contains a typo; the version stated here is the one implied by its proof (which is somewhat brief in \cite{LR}). In the next section we present a self-contained and different argument for sake of completeness. 
\begin{proposition}[Li--Radziwi\l\l's Lemma~1 with $\al=1$ and $\be=0$]
\label{prop:bounding_F}
For any positive integers $r,s$ with $(r,s)=1$ and $rs>1$, the following estimate holds uniformly for $T\leq t\leq 2T$:
\begin{equation*}
F(r,s,t)\ll_{\ep,\theta,G} \frac{(rs)^\ep}{\log{T}}
\end{equation*}
for any $\ep>0$.
\end{proposition}
Since $2\pi \bo\ell\cd\ba\geq 2\pi\min_{1\leq j\leq M}\ell_j>0$ uniformly for $\ba\in \z_+^M$, we have $e^{2\pi \bo\ell\cd\ba}>1$ for every $\ba$. Moreover, since $\left|\frac{r_\ba}{s_\ba}-e^{2\pi\bo\ell\cd\ba}\right|\leq e^{2\pi\bo\ell\cd\ba}/T^{1-\tilde{\eta}}$, we must have $r_\ba s_\ba>1$ for sufficiently large $T$. Applying Proposition~\ref{prop:bounding_F}, we deduce that
\begin{align*}
R_{B,-} &\ll_{\ep,\theta} \frac{T}{\log{T}}\sum_{\substack{\ba\in\z_+^M\\ s_{\ba}\leq K_{\ba}}}\frac{|c_{\ba}|}{(r_{\ba}s_{\ba})^{1/2-\varepsilon}}.
\end{align*}
Since $\ep>0$ was chosen in terms of $\theta$, all implicit constants depend on $\theta$ and on the fixed data $P$ and $G$. Using $r_{\ba}\asymp s_{\ba}e^{2\pi\bo\ell\cd\ba}$ gives
\begin{align*}
R_{B,-} &\ll_{\theta,P,G} \frac{T}{\log{T}}\sum_{\substack{\ba\in\z_+^M}} |c_{\ba}|e^{2\pi\bo\ell\cd\ba\left(-1/2+\ep\right)}
\ll_{\theta,P,G}\frac{T}{\log T}=o(T),
\end{align*}
and the result follows.

\section{Proof of Proposition \ref{prop:bounding_F}}
In this section, we give an independent proof of Proposition \ref{prop:bounding_F}. We begin with two auxiliary lemmas.
\begin{lemma}
\label{lem:H}
Let
\begin{equation*}
    H(x):=\sum_{q\geq1}\frac1qW\!\left(\frac{q^2}{x}\right).
\end{equation*}
Then, for every $A>0$ and $x\geq1$,
\[
H(x)=\frac12\log x+\gamma_0+O_{A,G}(x^{-A}),
\]
where \(\gamma_0\) is Euler's constant.
\end{lemma}
\begin{proof}
Recalling the definition of $W$ from \eqref{eq:W_def} and interchanging the sum and integral on the line $\Re w=c>0$ gives
\begin{align*}
    H(x)=\frac1{2\pi i}\int_{(c)}
 x^w\zeta(1+2w)G(w)\,\frac{dw}{w}.
\end{align*}
We then move the contour to $\Re w=-A$ and note that the only contributing pole is at $w=0$. Since
\begin{align*}
    \zeta(1+2w)=\frac1{2w}+\gamma_0+O(w),
 \quad
 x^w=1+w\log x+O(w^2),
\end{align*}
and $G(0)=1,\ G'(0)=0$, the residue is $\frac12\log x+\gamma_0$. The rapid vertical decay of $G$, together with the standard polynomial growth of $\zeta$ on vertical strips, bounds the new integral by $O_{A,G}(x^{-A})$.
\end{proof}

We now observe that coefficients $b_m$ of the mollifier $M_{\theta}(s)$ can be written in terms of Barban--Vehov sieve weights. To this end, for $X\geq1$, let $\rho_X(n):=\mu(n)\log(X/n)\1_{n\leq X}$ and note that  $b(n)=\rho_{T^\theta}(n)/\log{T^\theta}.$ We recall that $[m,n]$ denotes the least common multiple of $m$ and $n$, $\phi$ is Euler's totient function, $\omega(m)$ is the number of distinct prime divisors of $m$, and $\sigma(m)=\sum_{d\mid m} d$ is the sum of all positive divisors of $m$. 

\begin{lemma}
\label{lem:selberg}
Let $q\in\mathbb{N}$ be square-free. There exists an absolute constant $C_1$ such that, uniformly in $X,Y\geq1$,
\begin{equation}
\label{eq:mixed-selberg}
    \sum_{\substack{m,n\geq1\\ (mn,q)=1}} \frac{\rho_X(m)\rho_Y(n)}{[m,n]}=\frac{Q}{\phi(Q)}\log{\min{(X,Y)}}+O\left(C_1^{\omega(Q)}\right).
\end{equation}
Moreover, given $N\geq2$, suppose that $N^{9/10}\leq X,Y\leq N$ and $|C|\ll \log{N}$, with a fixed implicit constant. Then, for a constant
$C_1>1,$
\begin{align}
\label{eq:weighted-selberg}
 &\sum_{\substack{m,n\geq1\\(mn,q)=1}}
 \frac{\rho_X(m)\rho_Y(n)}{[m,n]}
 \left(
 C+\log(m,n)-\frac12\log m-\frac12\log n
 \right) \notag\\
 &\qquad =
 \frac{Q}{\phi(Q)}\left(
 C\log\min(X,Y)+\frac12\log X\log Y
 \right)
 +O\left(C_1^{\omega(Q)}\log{N}\right).
\end{align}
\end{lemma}
\begin{proof}
Indeed, \cite[Proposition~2.4]{Mu18} adapted to our context gives, for every $M>XYQ$,
\begin{equation}
\label{eq:Murty_lemma}
    S(M):=\sum_{\substack{k\leq M\\ k\equiv 1 \mod Q}} \sum_{m,n\mid k} \rho_X(m)\rho_Y(n)=\frac{M}{\phi(Q)} \log \min(X,Y)+O(XY)+O\left(\frac{M\sigma(Q)}{\phi(Q)^2}\right).
\end{equation}
The left-hand side above can be written as
\begin{align*}
    S(M)&=\sum_{\substack{m,n\geq1\\ (mn,Q)=1}} \rho_X(m)\rho_Y(n)\cd D, \text{ where }\\
    D:&=\#\left\{1\leq k\leq M; k\equiv 1 \mod Q \text{ and } [m,n]\mid k\right\}.
\end{align*}
Since $([m,n],Q)=1$, the Chinese remainder theorem gives $D=M/(Q[m,n])+O(1)$. Since $\sum_{m\leq X} |\rho_X(m)|\leq X\log X$, we  multiply both sides of \eqref{eq:Murty_lemma} by $Q/M$, and then let $M\ra\infty$ to arrive at 
\begin{align*}
    \sum_{\substack{m,n\geq1\\ (mn,Q)=1}} \frac{\rho_X(m)\rho_Y(n)}{[m,n]}=\frac{Q}{\phi(Q)}\log{\min{(X,Y)}}+O\left(\frac{Q\sigma(Q)}{\phi(Q)^2}\right).
\end{align*}
Because $Q$ is square-free, then $Q\sigma(Q)\phi(Q)^{-2}=\prod_{p\mid Q} p(p+1)(p-1)^{-2}\leq 6^{\omega(Q)}$, thus establishing \eqref{eq:mixed-selberg}.

For the weighted sum, we suppose by symmetry that \(X\leq Y\), and write \(x=\log X\), \(y=\log Y\), and \(E_Q=C_0^{\omega(Q)}\).
Partial summation
\[
 \rho_X(m)\log m
 =x\rho_X(m)-2\int_1^X\rho_u(m)\,\frac{du}{u}
\]
together with \eqref{eq:mixed-selberg} give
\begin{align*}
    \sum_{\substack{m,n\geq1\\ (mn,Q)=1}} \frac{\rho_X(m)\log(m)\rho_Y(n)}{[m,n]}&=O(C_1^{\omega(Q)}\log N), \text{ and}\\
    \sum_{\substack{m,n\geq1\\ (mn,Q)=1}} \frac{\rho_X(m)\rho_Y(n)\log(n)} {[m,n]}&=\frac{Q}{\phi(Q)}(x^2-xy)+O(C_1^{\omega(Q)}\log N).
\end{align*}
Moreover, since $\rho_X,\rho_Y$ are supported on square-free integers,
\begin{align*}
 \sum_{\substack{m,n\geq1\\(mn,Q)=1}}
 \frac{\rho_X(m)\rho_Y(n)\log(m,n)}{[m,n]}
 &=\sum_{\substack{p\leq X\\p\nmid Q}}\frac{\log p}{p}\,\sum_{\substack{m,n\geq1\\ (mn,Qp)=1}} \frac{\rho_{X/p}(m)\rho_{Y/p}(n)}{[m,n]}\\
 &=\frac{Q}{\phi(Q)}\sum_{\substack{p\leq X\\p\nmid Q}}
 \frac{\log p}{p-1}\log\frac{X}{p}+O(C_1^{\omega(Q)}\log N)\\
 &=\frac{c_Qx^2}{2}+O(C_1^{\omega(Q)}\log N).
\end{align*}
 Combining these estimates with \eqref{eq:mixed-selberg}, the main term is
\[
 \frac{Q}{\phi(Q)}\left[Cx+\frac{x^2}{2}-\frac{x^2-xy}{2}\right]
 =\frac{Q}{\phi(Q)}\left(Cx+\frac{xy}{2}\right).
\]
Recalling that \(|C|\ll\log N\), the total error is
\(O(C_1^{\omega(Q)}\log N)\) (after enlarging \(C_1\) if needed).  This proves \eqref{eq:weighted-selberg}.
\end{proof}

We are now ready to prove Proposition \ref{prop:bounding_F}.
\begin{proof}[Proof of Proposition \ref{prop:bounding_F}]
Without loss of generality, assume that $0<\varepsilon<1$. For fixed $h,k$, let $g:=(hr,ks)$. Since $(r,s)=1$, we note that
\begin{align*}
    \operatorname{lcm}\!\left(\frac{k}{(k,r)},\frac{h}{(h,s)}\right)
 =\frac{hk}{g}.
\end{align*}
Consequently, all solutions of $mk=rd$ and $nh=sd$ are given by the parametrization 
\begin{align*}
    d=q\frac{hk}{g},\qquad
 m=q\frac{hr}{g},\qquad
 n=q\frac{ks}{g},\qquad q\geq1.
\end{align*}
Since $b(n)=\rho_{T^\theta}(n)/\log{T^\theta}$, we can write
\begin{equation}\label{eq:F-H}
 F(r,s,t)=\frac1{(\log{T^\theta})^2}\sum_{h,k\leq T^\theta}
 \rho_{T^\theta}(h)\rho_{T^\theta}(k)\frac{g}{hk}
 H\!\left(\frac{tg^2}{2\pi hkrs}\right).
\end{equation}
We choose $0<\eta<1$, depending only on $\theta$, as specified below. We split the argument into two cases.
\subsection*{The case $rs>T^\eta$} Since $|b(n)|\leq1$, the rapid decay of $W$ gives, for every $\delta>0$ and all sufficiently large $A$,
\begin{align*}
 |F(r,s,t)|
 &\ll_A
 \sum_{d\geq1}\frac{\tau(rd)\tau(sd)}d
 \left(1+\frac{rsd^2}{T^{1+2\theta}}\right)^{-A}\\
 &\ll_\delta
 (rs)^\delta
 \sum_{d\geq1}d^{-1+2\delta}
 \left(1+\frac{rsd^2}{T^{1+2\theta}}\right)^{-A}\\
 &\ll_\delta
 (rs)^\delta
 \left\{1+\left(\frac{T^{1+2\theta}}{rs}\right)^\delta\right\}\ll_{\ep,\theta,G} (rs)^{\ep/2},
\end{align*}
provided that we choose $\delta$ so that $\delta\frac{1+2\theta}{\eta}<\frac{\varepsilon}{2}$. Since $rs>T^\eta$, it follows that $\log T\ll_{\varepsilon,\eta}(rs)^{\varepsilon/2}$, and hence $|F(r,s,t)|\ll_{\ep,\theta,G} (rs)^\ep/\log{T}$, as desired.

\subsection*{The case $rs\leq T^\eta$} We let $N=T^\theta$. Since $\frac{tg^2}{2\pi hkrs}\gg T^{1-2\theta-\eta}$, Lemma \ref{lem:H} gives
\begin{align}
\label{eq:F_after_H_lemma}
    (\log{T^\theta})^2F(r,s,t)=\sum_{h,k\leq N}
 \rho_{T^\theta}(h)\rho_{T^\theta}(k)\frac{g}{hk}\left[\frac{1}{2}\log\left(\frac{tg^2}{2\pi hkrs}\right)+\ga_0\right]+O_\theta(T^{-100}).
\end{align}
Indeed, $g\leq\sqrt{rshk}$ and therefore
\begin{align*}
    \sum_{h,k\leq N}\frac{g}{hk}
 \ll\sqrt{rs}
 \left(\sum_{n\leq N}n^{-1/2}\right)^2
 \ll\sqrt{rs}\, T^\theta.
\end{align*}
Provided that $\eta<1-2\theta$, choosing $A$ large enough that $(1-2\theta-\eta)A>101+\theta+\eta/2$ gives the error $O_\theta(T^{-100})$.

For an integer $m$, consider its radical $\operatorname{rad}(m):=\prod_{p\mid m} p$. Let $R:=\operatorname{rad}(r)$, $S:=\operatorname{rad}(s)$, and $Q:=RS=\operatorname{rad}(rs)$. Since we may assume that $h$ and $k$ are square-free, we can write them uniquely as $h=ah_0$ and $k=bk_0$, where $a,b\mid Q$ and $(h_0k_0,Q)=1$. Set $a_R:=(a,R)$, $a_S:=(a,S)$, $b_R:=(b,R)$, and $b_S:=(b,S)$. We note that
\begin{equation}\label{eq:gcd-split}
 (hr,ks)=(h_0,k_0)a_Sb_R, \text{ and hence }
 \frac{g}{hk}=\frac{(hr,ks)}{hk}
 =\frac1{[h_0,k_0]a_Rb_S}.
\end{equation}
Since $\rho_N(h)=\mu(a)\rho_{N/a}(h_0)$ and $\rho_N(k)=\mu(b)\rho_{N/b}(k_0)$, equation \eqref{eq:gcd-split} allows us to rewrite \eqref{eq:F_after_H_lemma} as
\begin{align}
\label{eq:local-form}
 (\log N)^2F(r,s,t)
 &=
 \sum_{a,b\mid Q}\frac{\mu(a)\mu(b)}{a_Rb_S}
 \sum_{\substack{m,n\geq1\\(mn,Q)=1}}
 \frac{\rho_{N/a}(m)\rho_{N/b}(n)}{[m,n]}
 \notag\\
 &\quad\times
 \left(
 C_{a,b}+\log(m,n)-\frac12\log m-\frac12\log n
 \right)
 +O(T^{-100}),
\end{align}
where $C_{a,b}:=\frac12\log\frac{t}{2\pi rs}+\gamma_0+\frac12\log\frac{a_Sb_R}{a_Rb_S}$. Observe that $|C_{a,b}|\ll \log N$ and moreover, if we choose $\eta<\theta/10$, then $Q\leq rs\leq T^{\eta}<N^{1/10}$, and hence $N^{9/10}\leq N/a,N/b\leq N$.

We are now is the position to apply Lemma \ref{lem:selberg}. Indeed, we apply Lemma \ref{lem:selberg} to every pair $a,b$ to arrive at
\begin{equation}\label{eq:F-main}
 (\log N)^2F(r,s,t)=c_Q\mathcal M+O\left(C_2^{\omega(Q)}\log N\right)+O(T^{-100}),
\end{equation}
where the factor $\tau(Q)^2=4^{\omega(Q)}$ coming from the number of pairs $(a,b)$ has been absorbed by enlarging $C_2$, and
\begin{equation}\label{eq:M-def}
 \mathcal M:=
 \sum_{a,b\mid Q}\frac{\mu(a)\mu(b)}{a_R b_S} \left[
 C_{a,b}\min\left(\log \frac{N}{a},\log \frac{N}{b}\right)
 +\frac{1}{2}\log \frac{N}{a}\log \frac{N}{b}
 \right].
\end{equation}

It remains to exhibit the cancellation in $\mathcal M$. To this end, we use pointwise bound
\begin{align*}
    \sum_{a,b\mid Q}\frac{\mu(a)\mu(b)}{a_R b_S} \left[\ga_0+\frac{1}{2}\log\left(\frac{a_S b_R}{a_R b_S}\right)\right]\min\left(\log \frac{N}{a},\log \frac{N}{b}\right)\ll \tau(Q)^2\log Q\log N,
\end{align*}
owing to the estimate
\[
\left|\log\left(\frac{a_S b_R}{a_R b_S}\right)\right|\leq 4\log Q
\]
and $0\leq \log(N/a),\log(N/b)\leq\log N$.

For the remaining term in $C_{a,b}$, we need a somewhat finer estimate because $|\log(t/(2\pi rs))|\gg\log T^{1-\eta}$. For $0\leq u\leq \log N$, define
\begin{align*}
   A(u):=\sum_{\substack{a\mid Q\\u\leq \log \frac{N}{a}}} \frac{\mu(a)}{a_R}\, \text{ and } \,
   B(u):=\sum_{\substack{b\mid Q\\u\leq \log \frac{N}{b}}} \frac{\mu(b)}{b_S}.
\end{align*}
We also consider the corresponding sums over all divisors of $Q$. Expanding into Euler products, we have
\begin{align*}
    A_*&=\sum_{a|Q}\frac{\mu(a)}{a_R}= \prod_{p\mid R}\left(1-\frac1p\right)
 \prod_{p\mid S}(1-1),\\
    B_*&=\sum_{b|Q}\frac{\mu(b)}{b_S}=
 \prod_{p\mid R}(1-1)
 \prod_{p\mid S}\left(1-\frac1p\right).
\end{align*}
Since $Q>1$, at least one of $R,S$ is nontrivial, and therefore $A_*B_*=0$. The elementary identity
\begin{align*}
    \min\left(\log \frac{N}{a},\log \frac{N}{b}\right)
 =\int_{0}^{\log N}\1_{u\leq \log \frac{N}{a}}\1_{u\leq \log \frac{N}{b}}\,\d u
\end{align*}
implies
\begin{equation}
\label{eq:min-integral}
 \sum_{a,b\mid Q} \frac{\mu(a)\mu(b)}{a_R b_S} \min\left(\log \frac{N}{a},\log \frac{N}{b}\right)
 =
 \int_{0}^{\log{N}} A(u) B(u)\,\d u.
\end{equation}
If $u\leq \log (N/Q)$, every divisor of $Q$ occurs in both partial sums; the integrand is then $A_*B_*=0$. Since $|A(u)|,|B(u)|\leq\tau(Q)$ and $0\leq \log(N/a),\log(N/b)\leq\log N$, we obtain
\begin{align*}
 \left|\sum_{a,b\mid Q}\frac{\mu(a)\mu(b)}{a_R b_S} \min\left(\log \frac{N}{a},\log \frac{N}{b}\right) \right|\leq\int_{\log (N/Q)}^{\log N}|A(u)B(u)|\,\d u \leq\tau(Q)^2\log Q.
\end{align*}
Combining this with $|\log(t/(2\pi rs))|\ll \log N$ yields an error term $O(\tau(Q)^2\log Q\log N)$.

For the last term in \eqref{eq:M-def}, we factor the sum:
\begin{align*}
 \sum_{a,b\mid Q} \frac{\mu(a)\mu(b)}{a_R b_S}\log \frac{N}{a} \log \frac{N}{b}=\left(
A_*\log N-\sum_{a\mid Q}\frac{\mu(a)}{a_R}\log a
 \right)
 \left(
 B_*\log N-\sum_{b\mid Q}\frac{\mu(b)}{b_S}\log b
 \right).
\end{align*}
Upon expanding this product, its $(\log N)^2$ term vanishes because $A_*B_*=0$. Since
\begin{align*}
    \sum_{a\mid Q}\frac{|\mu(a)|}{a_R}\log a,\ 
 \sum_{b\mid Q}\frac{|\mu(b)|}{b_S}\log b
 \leq\tau(Q)\log Q
\end{align*}
and $\log Q<\log N/10$, the remaining terms contribute $O\left(\tau(Q)^2\log N\log Q\right)$.

Therefore, $\mathcal M\ll\tau(Q)^2\log N\log Q$ and after replacing $\ep>0$ by a smaller number if necessary, we obtain
\begin{align*}
    (\log N)^2|F(r,s,t)|\ll_{\varepsilon,\theta,G}Q^\varepsilon \log N+\frac{1}{T^{100}}.
\end{align*}
Since $Q\leq rs$ and $\theta>1/10$, we conclude that
\begin{align*}
    |F(r,s,t)|
 \ll_{\varepsilon,\theta,G}
 \frac{(rs)^\varepsilon}{\log T}.
\end{align*}
and the result follows.
\end{proof}

\section{Proof of Theorem \ref{thm:main2}}

The proof of Theorem \ref{thm:main2} is inspired by the resonance method of~\cite[Section 5]{LR}, however our arguments are substantially different due to extra complications. We note that direct substitution of the Kurasov--Sarnak formula for Poisson summation is insufficient since the positive Kurasov--Sarnak spectrum is too large (more precisely, it has $O(N^M)$ frequencies up to $N$, rather than $O(N)$). We therefore treat the low and high frequencies using different approaches.

Throughout this section, we fix a real-rooted trigonometric polynomial $P$, its associated vector $\bo\ell$ from Theorem \ref{thm:KS}, and a nonzero function $\p\in C_c^\infty((1,2))$ such that $0\leq \p\leq 1$. Since several of the following computations already appeared in the estimates for the first and second moments, we focus on the new features that arise.

\subsection*{Step 1: Construction of the resonator}
We choose $0<\eta<1/120$ and set $\delta=1/3+20\eta$,
\begin{equation}\label{eq:parameters}
 X=T^{1/2-\delta},\qquad H=T^{1/2-\eta},\qquad
 \mathcal L=\sqrt{\log X\log\log X},\qquad
 R=(\log\mathcal L)^3,
\end{equation}
where \(X\) is the length of the resonator.

We first describe a small exceptional set of primes. Let
\[
 \Xi:=\{\pm\bo\ell\cd\ba:\ba\in\z_+^M\}
\]
be the set of signed nonzero crystalline frequencies. For each $\xi\in\Xi$ with $0<|\xi|\leq R$, consider coprime positive integers $a,b$ satisfying
\begin{equation}\label{eq:low-pair}
 a\leq X,
 \qquad b\leq H,
 \qquad
 \left|\xi-\frac1{2\pi}\log\frac ab\right|
 \leq \frac{1}{T^{1-\eta}}.
\end{equation}
There is at most one such pair since distinct fractions in these ranges have logarithms separated by \(\gg(XH)^{-1}=T^{-1+\delta+\eta}\). For every value of $a$ or $b$ that is not $1$, choose one of its prime divisors, and denote the resulting set by \(\mathcal S\). Thus $\#\mathcal S\ll R^M$.

If a pair in \eqref{eq:low-pair} has $a=1$, we include the prime chosen from $b$ in a square-free integer $Q$. Since in that case $\log b\leq 2\pi R+o(1)$, we have the esimates
\begin{equation}\label{eq:Q-size}
 \omega(Q)\ll R^M,
 \qquad
 \log Q\ll R^{M+1}=o(\log T).
\end{equation}
If there are no such pairs, we simply take $Q=1$.

Define a nonnegative multiplicative function $r$, with $r(1)=1$ and supported on square-free integers, by
\begin{equation}\label{eq:r-def}
 r(p)=
 \begin{cases}
 \displaystyle\frac{\mathcal L}{\sqrt p\log p},
 &\mathcal L^2\leq p\leq
       \exp((\log\mathcal L)^2),\quad p\notin\mathcal S,\\[6pt]
 0,&\text{otherwise},
 \end{cases}
\end{equation}
and $r(p^j)=0$ for $j\geq2$.

The resonator takes the form
\begin{equation}\label{eq:B-def}
 B(s)=\sum_{n\leq X}\frac{\sqrt n\,r(n)}{n^s},
 \qquad
 B\!\left(\tfrac12+it\right)=\sum_{n\leq X}r(n)n^{-it}.
\end{equation}
Finally, we put
\begin{equation}\label{eq:AQ-def}
 A_Q(s)=\sum_{\substack{h\leq TQ\\(h,Q)=1}}h^{-s},
 \qquad
 P_Q(s)=\prod_{p\mid Q}(1-p^{-s}).
\end{equation}

We note in passing that instead of the usual truncation of $\zeta(s),$ we will use the coprime Dirichlet polynomial $A_Q(s)$ of length $TQ$. The reason is that if a low spectral relation has reduced numerator $a>1$, the prime removed from $a$ makes the corresponding resonator coefficient zero. If $a=1$, a chosen prime divisor of $b$ could occur in the zeta variable $h$ rather than in the resonator variable. The condition $(h,Q)=1$ removes precisely this remaining possibility.

\subsection*{Step 2: The sampling lemma}

Let
\begin{equation}\label{eq:E-S-def}
 E_X=\sum_{n\leq X}r(n)^2 \qquad \text{ and }
 \qquad
 S_X=\sum_{nh\leq X}\frac{r(nh)r(n)}{\sqrt h}.
\end{equation}

\begin{lemma}\label{lem:sampling}
With the preceding choices, as $T\to\infty$ with $\eta$ fixed,
\begin{align}
 \sum_{P(\gamma)=0}
  \left|B\!\left(\tfrac12+i\gamma\right)\right|^2
  \p\!\left(\frac\gamma T\right)
 &=c_0T\ft{\p}(0)E_X(1+o(1)),
 \label{eq:sample-den}\\
 \sum_{P(\gamma)=0}
  A_Q\!\left(\tfrac12+i\gamma\right)
  \left|B\!\left(\tfrac12+i\gamma\right)\right|^2
  \p\!\left(\frac\gamma T\right)
 &=c_0T\ft{\p}(0)S_X(1+o(1)).
 \label{eq:sample-num}
\end{align}
The constants implicit in $o(1)$ may depend on $P$, $\p$, and $\eta$.
\end{lemma}

\begin{proof}
We give the details for~\eqref{eq:sample-num}; the same
argument with $h=1$ proves \eqref{eq:sample-den}.

Expanding the Dirichlet polynomials and applying the Kurasov--Sarnak summation formula gives the main term
\begin{equation}\label{eq:continuous-piece}
 c_0\int_{\r}A_Q\!\left(\tfrac12+it\right)
       \left|B\!\left(\tfrac12+it\right)\right|^2
       \p\!\left(\frac tT\right)\,\d t
\end{equation}
 together with nonzero-frequency terms of the form
\begin{equation}\label{eq:offdiag-KS}
 T c_{\ba}
 \sum_{\substack{m,n\leq X\\h\leq TQ\\(h,Q)=1}}
 \frac{r(m)r(n)}{\sqrt h}\,
 \ft{\p}\!\left(
 T\left(\bo\ell\cd\ba-\frac1{2\pi}\log\frac{m}{nh}\right)
 \right),
 \qquad \pm\ba\in\z_+^M.
\end{equation}
Indeed, the diagonal
$m=nh$ equals $c_0T\ft{\p}(0)S_X$ and the condition $(h,Q)=1$ holds automatically. If $m\neq nh$ and $nh\leq2X$, then $|\log(m/(nh))|\gg X^{-1}$; if $nh>2X$, that logarithm is bounded away from zero. The rapid decay of $\ft{\p}$ therefore shows that th off-diagonal contribution from the main term is $o(TS_X)$.

For \eqref{eq:offdiag-KS}, we separate the sum according to the distance to $\ell\cd\ba$ . The rapid decay of $\ft{\p}$, combined with $0\leq r(n)\leq1$ and $Q=T^{o(1)}$, implies the contribution of the part for which
\begin{equation}\label{eq:near-condition}
 \left|\bo\ell\cd\ba-\frac1{2\pi}\log\frac{m}{nh}\right|>T^{-1+\eta}
\end{equation}
is $o(TE_X)$. Thus, it remains to consider the terms satisfying
\[
 \left|\bo\ell\cd\ba-\frac1{2\pi}\log\frac{m}{nh}\right|\leq T^{-1+\eta},
\]
which restricts the frequencies $\bo\ell\cd\ba$ to $|\bo\ell\cd\ba|\ll\log T$.

\subsection*{The terms with $nh\leq H$} 
Reduce $m/(nh)$ to $a/b$, so that $a\leq X$ and $b\leq H$. The spacing argument used after \eqref{eq:low-pair} shows that, for each signed frequency $\xi=\bo\ell\cd\ba\in\Xi$, there is at most one such pair.

We first deal with the case $0<|\xi|\leq R$. This is where we use the device described above. If $a>1$, the selected prime divisor of $a$ belongs to $\mathcal S$. Since $a\mid m$, this gives $r(m)=0$. If $a=1$, let $p\mid b$ be the selected prime. It belongs to both $\mathcal S$ and $Q$. A nonzero summand must satisfy $r(m)\neq0$ and $r(n)\neq0$. Since $nh=bm$ and $p$ was removed from the support of $r$, this implies $p\nmid n$; hence $p\mid h$, contrary to $(h,Q)=1$. Thus, all the low spectral frequencies in the range $nh\leq H$ contribute zero.


For the higher frequencies and for coprime $a,b,$ we let
\[
 U(a,b)=
 \sum_{\substack{m,n\leq X,\ h,d\geq1\\m=ad,\ nh=bd}}
 \frac{r(m)r(n)}{\sqrt h},
\]
and define
\begin{equation}\label{eq:K-def}
 K(b)=b^{-1/2}
 \prod_{\substack{p\mid b\\p\in\supp(r)}}
       (1+r(p)\sqrt p).
\end{equation}
We claim that
\begin{equation}\label{eq:local-correlation}
 U(a,b)\leq r(a)K(b)S_X.
\end{equation}
To see this, put $u=(n,b)$, $n=un_1$, and $b=ub_1$. Since $(n_1,b_1)=1$, the relation $nh=bd$ implies $d=n_1k$ and $h=b_1k$. Whenever the corresponding summand is nonzero, the multiplicativity gives
\[
 \frac{r(m)r(n)}{\sqrt h}
 =\frac{r(a)r(u)r(n_1)^2r(k)}{\sqrt{b_1k}}.
\]
The sum over $n_1,k$ is a subsum of $S_X$, and
\[
 \sum_{\substack{u\mid b\\u\in\supp(r)}}
 \frac{r(u)}{\sqrt{b/u}}
 =K(b),
\]
which proves~\eqref{eq:local-correlation}.

We now put $\kappa:=\frac{\log\log\mathcal L}{10\log\mathcal L}$. We use several direct bounds
\begin{equation}\label{eq:prime-local-bounds}
 r(p)\leq p^{-\kappa},
 \qquad
 p^{-1/2}(1+r(p)\sqrt p)\leq p^{-\kappa}
\end{equation}
when $T$ is sufficiently large. More explicitly, if $p^j\Vert b$ and $r(p)\neq0$, then
\[
 p^{-j/2}(1+r(p)\sqrt p)\leq p^{-j\kappa};
\]
if $r(p)=0$, its local factor in $K(b)$ is simply $p^{-j/2}\leq p^{-j\kappa}$. Primes $p$ with $r(p)=0$ also force $r(a)=0$ when they divide $a$. Consequently,
\begin{equation}\label{eq:ab-decay}
 r(a)K(b)\leq(ab)^{-\kappa}.
\end{equation}
and our distance condition further implies that $ab\geq \frac12 e^{2\pi|\bo\ell\cd\ba|}$.

It follows from~\eqref{eq:local-correlation}--\eqref{eq:ab-decay} and partial summation that
\begin{align}
 \sum_{\substack{\pm\ba\in\z_+^M,\ |\bo\ell\cd\ba|>R\\
                  \text{a near pair }(a_{\bo\ell\cd\ba},b_{\bo\ell\cd\ba})\text{ exists}}}
 |c_\ba|U(a_{\bo\ell\cd\ba},b_{\bo\ell\cd\ba})
 &\ll S_X
 \sum_{\substack{\pm\ba\in\z_+^M,\ |\bo\ell\cd\ba|>R}}
 |c_\ba| e^{-2\pi\kappa|\bo\ell\cd\ba|}
 =o(S_X),                                      \label{eq:high-tail}
\end{align}
because $\kappa R\asymp (\log\mathcal L)^2\log\log\mathcal L\longrightarrow\infty$.

\subsection*{The terms with $nh>H$}
Since $n\leq X$, the condition $nh>H$ implies $h>H/X=T^{\delta-\eta}$. For fixed $m,n$, and $\bo\ell\cd\ba$, the spacing condition forces $h$ to lie in a multiplicative interval of length $O(T^{-1+\eta})$. Since $h\leq TQ=T^{1+o(1)}$, this interval contains at most $T^{\eta+o(1)}$ integers and therefore, using the bounds $(\sum_{n\leq X}r(n))^2\leq XE_X$ and $\delta=1/3+20\eta,$ the total contribution of this range is
\begin{align}
 &\ll
 T\,T^{-\delta/2+3\eta/2+o(1)}
 \left(\sum_{n\leq X}r(n)\right)^2
 \sum_{\substack{\pm\ba\in\z_+^M\\|\bo\ell\cd\ba|\ll\log T}}|c_\ba|
 \notag\\
 &\ll
 T^{3/2-3\delta/2+3\eta/2+o(1)}E_X
 =O\!\left(T^{1-c\eta}E_X\right)              \label{eq:large-h}
\end{align}
which gives acceptable contribution of $o(TS_X)$.
\end{proof}

\subsection*{Collecting the pieces together}
Here we briefly comment on how to finish the proof of our theorem. The calculation using Rankin's trick due to Soundararajan \cite[Theorem~2.1]{Sound}, in the form used by Li--Radziwi\l\l\ \cite[pp.~345--346]{LR}, gives
\begin{equation}\label{eq:quotient}
 \frac{S_X}{E_X}
 \geq\exp\!\left((1+o(1))
       \sqrt{\frac{\log X}{\log\log X}}\right).
\end{equation}
Deleting \(\mathcal S\) does not affect this asymptotic since its total contribution to the logarithm of the Euler product is
\[
 \ll\frac{R^M}{\mathcal L\log\mathcal L}=o(1).
\]
The standard truncation formula for \(\zeta\) \cite[Theorem~4.11]{Titch} gives, uniformly for \(s=1/2+it\), \(T\leq t\leq2T\),
\begin{equation}\label{eq:AQ-zeta}
 A_Q(s)=\zeta(s)P_Q(s)
 -\frac{(TQ)^{1-s}}{s-1}\frac{\varphi(Q)}Q
 +O\!\left(\tau(Q)(TQ)^{-1/2}\right).
\end{equation}
By \eqref{eq:Q-size}, the last two terms are \(T^{-1/2+o(1)}\) and moreover
\begin{equation}\label{eq:PQ}
 \sup_t|P_Q(\tfrac12+it)|
 \leq\prod_{p\mid Q}(1+p^{-1/2})
 \leq\exp(O_P(R^M))
 =\exp\!\left(o\!\left(
       \sqrt{\frac{\log T}{\log\log T}}\right)\right).
\end{equation}

Denote the left-hand sides of \eqref{eq:sample-den} and \eqref{eq:sample-num} by \(\mathcal M_0\) and \(\mathcal M_1\), respectively. From \eqref{eq:sample-den}, \eqref{eq:sample-num}, and \eqref{eq:quotient},
\[
 \frac{|\mathcal M_1|}{\mathcal M_0}
 \geq\exp\!\left((1+o(1))
       \sqrt{\frac{\log X}{\log\log X}}\right).
\]
On the other hand, positivity of \(|B|^2\p\) and
\eqref{eq:AQ-zeta} show that this ratio is at most
\[
 \sup_{\substack{T\leq\gamma\leq2T\\P(\gamma)=0}}
 |\zeta(\tfrac12+i\gamma)|
 \sup_t|P_Q(\tfrac12+it)|+T^{-1/2+o(1)}.
\]
Using \eqref{eq:PQ} and \(X=T^{1/6-20\eta}\), we conclude that
\[
 \log\sup_{\substack{T\leq\gamma\leq2T\\P(\gamma)=0}}
 |\zeta(\tfrac12+i\gamma)|
 \geq\left(\sqrt{\frac16-20\eta}+o(1)\right)
       \sqrt{\frac{\log T}{\log\log T}}.
\]
Letting first \(T\to\infty\) and then \(\eta\downarrow0\) concludes the proof,

\subsection*{Acknowledgments}
B.K. and O.K. greatly enjoyed the research visits to IMPA in October 2024, where most of the content of the present work has been conceived.
F.G. acknowledges support from the following funding agencies: The Office of Naval Research GRANT14201749, The Serrapilheira Institute (Serra-2211-41824), FAPERJ (E-26/210.245/2024) and CNPq (309910/2023-4). B.K. was supported by  Australian Research Council Grants DE220100859 and DP230100534. G.V.  was supported by CNPq (141446/2023-4) and by a Fulbright Doctoral Dissertation Research Award. Part of this work was carried out during his visit to the University of Texas at Austin under the latter award; he gratefully acknowledges the Fulbright Program for its support and UT Austin for its hospitality and institutional support. He also thanks Lior Alon for fruitful discussions and valuable comments that improved the manuscript, and MIT for its hospitality during a separate research visit.


\subsection*{Declaration on the Use of Generative AI}
Generative AI was used to assist with paper editing, calculation verification, and with the proof simplification of Proposition \ref{prop:bounding_F} (this proposition is borrowed from the work \cite{LR}, but the proof there is somewhat compressed; rather than citing it directly, we have therefore decided to include a different argument for completeness) as well as locating relevant references. The authors assume full responsibility for the final content.


\end{document}